\documentclass[12pt]{article}

\usepackage{amsmath,amsthm,amsfonts,amssymb,amscd}
\usepackage{xy}
\input xy
\xyoption{all}

\allowdisplaybreaks

\DeclareMathOperator{\Pic}{Pic}
\DeclareMathOperator{\Coker}{Coker} 
\DeclareMathOperator{\Norm}{Norm} 
\newcommand{\Z}{\mathbb{Z}}
\newcommand{\F}{\mathbb{F}}
\newcommand{\A}{\mathbb{A}}
\newcommand{\Q}{\mathbb{Q}}
\newcommand{\V}{\mathcal{V}}
\newtheorem{thm}{Theorem}[section]
\newtheorem{prop}[thm]{Proposition}
\newtheorem{lemma}[thm]{Lemma}
\newtheorem{cor}[thm]{Corollary}
\newtheorem{con}[thm]{Conjecture}

\theoremstyle{definition}

\newtheorem{rem}[thm]{Remark}

\begin{document}

\title{\bf On Kervaire--Murthy conjecture, Bernoulli and Iwasawa numbers,
and zeroes of $p$-adic $L$-function}

\author{\Large A. Stolin \\
Department of Mathematical Sciences\\
Chalmers University of Technology and\\ University of Gothenburg\\
Gothenburg, Sweden\\
E-mail: alexander.stolin@gu.se}
\date{}

\maketitle

\begin{abstract}
The aim of the present paper is to establish relations between Iwasawa and Bernoulli numbers
based on some results by M. Kervaire and M. P. Murthy about the structure of the $K_0$ groups
of the integer group rings
of cyclic groups of prime power order $p^n .$ In particular, we will prove that
\begin{itemize}
  \item  $\lambda_{i}\leq
p-1$ under assumption that the generalized Bernoulli number $B_{1,\omega^{-i}}$ is not divisible by $p^2$. Here  $\omega$ is
the Teichm\"{u}ller character of ${\Z}/(p-1){\Z}$.
  \item $\lambda_{i}=1$ if  $B_{1,\omega^{-i}}$ is  divisible by $p^2$.
  \item We will prove that $S_{n,i}\cong \Z/(p^{n+k_i})$, where $S_n$ is the Sylow
  $p$-subgroup of the class group of the field
  $\Q(\zeta_n)$.
  Here, $\zeta_n$ is
  a primitive $p^{n+1}$-root of unity,
  $\varepsilon_{i}$ are idempotents
in the group ring ${\Z}_{p}[{\rm Gal}(\Q (\zeta_0) /\Q)]$, $S_{n,i}=\varepsilon_i (S_n)$,
and $k_i$ is the $p$-adic valuation of $B_{1,\omega^{-i}}$.
\item At the end we will prove that $k_i \leq 1$ and also $v_p (L_p (0, \omega^j))\leq 1$ for
even $j$ under certain conditions on zeroes of $L_p (0, \omega^j) .$
\item Throughout the paper we assume that $p$ satisfies Vandiver's conjecture.

\end{itemize}
%{\em Keywords:}  \\
\end{abstract}

\section{Introduction}
%AAA
Let $C_{n}$ denote the cyclic group of order $p^{n}$, where $p$ is an odd prime.
Let ${\Z} C_{n}$ be the integral group ring of $C_{n}$.

In this paper we study $\Pic$ ${\Z} C_{n}$ and some other groups related to
it, in particular, the ideal class group $C(F_{n})$ of the cyclotomic field
$F_{n}=\mathbb{Q}(\zeta_{n})$, where $\zeta_{n}$ is a primitive $p^{n+1}$-st
root of unity.

Throughout this paper we assume that $p$ is semi-regular, that is $p$ does
not divide the order of the ideal class group of the maximal real subfield
$F^{+}_{0}=\mathbb{Q}(\zeta_{0}+\zeta_{0}^{-1})$ in $F_{0}$. Let $A$ be an
abelian group. The following notation will be used in our paper:

\begin{itemize}
\item $F_{n}$ has been already defined, $F^{+}_{n}=\mathbb{Q}(\zeta_{n}+\zeta_{n}^{-1})$;
\item $\F_p =\Z/p\Z$;
\item $N\cdot A$ (or sometimes $A^N$, if it is clear from a context) is the direct sum of $N$ copies of $A$;
\item $dA$ or $A^{d}$ (depending on additive or multiplicative operation
on $A$) stands for the subgroup of $A$ which consists of the elements of the
form $da$ or $a^{d}$;
\item $A^{(d)}$ stands for the subgroup of $A$ which consists of the
elements of $A$ such that $da=0$ or $a^{d}=1$;
\item $A_{(p)}$ denotes the Sylow $p$-component of $A$. For $A=C(F_{n})$
we use a special notation $C(F_{n})_{(p)}=S(F_{n})=S_n$;
\item if $R$ is a commutative ring, then $U(R)$ denotes the group of units of $R$.
\item in the special case $R={\Z}[\zeta_{n}]$, we use $E_{n}$ for
$U({\Z}[\zeta_{n}])$;
\item further, we use notation $E_{n,k}$ for the subgroup of
$E_{n}$ consisting of units which are congruent to 1 modulo
$\mu^{k}_{n}=(1-\zeta_{n})^{k}$.
\end{itemize}

%{\bf Introduction}\\

Following \cite{KM} let us consider the fibre product diagram
$$
\xymatrix{{\Z}C_{n+1} \ar[r]^{i_{2}}
\ar[d]^{i_{1}} &
{\Z}[\zeta_{n}] \ar[d]^{j_{2}} \\
{\Z}C_{n} \ar[r]^{j_{1}}  &
\frac{{\F}_{p}[x]}{(x-1)^{p^{n}}}:=R_{n}}
$$
with obvious maps $i_{1}, i_{2}, j_{1}, j_{2}$.
The corresponding Mayer-Vietoris exact sequence can be written as follows:
$$
U({\Z}C_{n}) \times E_{n} \stackrel{j}{\longrightarrow}
U(R_{n}){\longrightarrow} \Pic ({\Z}C_{n+1}) {\longrightarrow} \Pic
({\Z}C_{n}) \times C(F_{n}) {\longrightarrow} 0.
$$
One of the main problems in computing $\Pic ({\Z}C_{n+1})$ is thus to
evaluate the co\-kernel $W_{n}$ of the map $j:U({\Z}C_{n})\times
E_{n}{\longrightarrow} U(R_{n})$

Instead of $W_{n}$ we will evaluate a bigger group
$$\mathcal{V}_{n}=\Coker\{j_{2}:E_{n}{\longrightarrow} U(R_{n})\}.$$
Clearly, $W_{n}$ is a factorgroup of $\mathcal{V}_{n}$.

In the calculation of $\mathcal{V}_{n}$ a decisive role will be played by
the action $G_{n}={\rm Gal}(F_{n}/\mathbb{Q})$ on the various rings involved in
the paper. Let\newline
$\delta : G_{n} {\longrightarrow} U({\Z}/p^{n+1}{\Z})$ be the
canonical isomorphism defined by $s(\zeta_{n})=\zeta^{\delta(s)}_{n}$, $s\in
G_{n}$. We will denote by $x_{n}$ the generator in ${\Z}[x]/(x^{p^{n}-1}) =
{\Z}C_{n}$ and in ${\F}_{p}[x]/(x-1)^{p^{n}}=R_{n}$ that corresponds to $x$.
Since $\delta(s)$ is an integer modulo $p^{n+1}$, prime to $p$, it is clear
that both $x^{\delta(s)}_{n+1}$ and $x^{\delta(s)}_{n}$ are well-defined.
Moreover, the maps in the fibre product above commute with the action of
$G_{n}$. Let $c\in G_{n}$ be the complex conjugation. It is clear that
$\mathcal{V}_{n}=\mathcal{V}^{+}_{n}\times \mathcal{V}^{-}_{n}$, where
$\mathcal{V}_{n}^{+}$ consists of elements such that $c(a)=a$ and
$\mathcal{V}^{-}_{n}$ consists of elements such that $c(a)=a^{-1}$ (we take
into account that $\mathcal{V}_{n}$ is a $p$-group). Similarly,
$W_{n}=W^{+}_{n}\times W^{-}_{n}$. For any abelian group $A$, let us denote
by $A^{\ast}$ the group of characters of $A$.

The main results proved by Kervaire and Murthy in \cite{KM} was

\begin{thm}
%{\bf Theorem 1.1}
If $p$ is a semi-regular odd prime, then
$$
(W^{+}_{n})^{\ast} \subseteq (\mathcal{V}^{+}_{n})^{\ast} \subseteq
S^{-}(F_{n-1})=S(F_{n-1}) =: S_{n-1}.
$$
In other words, there is a surjection $S_{n-1}^{\ast}\to \mathcal{V}_n^+$
\end{thm}

They also conjectured that, in fact, $W^{+}_{n} \cong \mathcal{V}^{+}_{n}
\cong S^{\ast}_{n-1}$. The first main result of our paper is a weak version
of the Kervaire and Murthy conjecture, namely
$$
(S_{n-1})^{(p)} \cong (\mathcal{V}^{+}_{n} /
(\mathcal{V}^{+}_{n})^{p})^{\ast} =((\mathcal{V}^{+}_{n})^{\ast})^{(p)}
$$
Another important result proved in this paper (which gives a new link between
the class groups and the groups $\mathcal{V}_n$)  is that there exists a
canonical embedding $$S_{n-1}^{(p)}\to \mathcal{V}_n^{-} /(\mathcal{V}_n^-)^p $$
Working on the Kervaire and Murthy  conjecture, Ullom proved in \cite{U} that under certain assumptions
on the Iwasawa numbers $\lambda_i$ explained later, the group $W^+_n $ can be described
as follows:
$$ W^{+}_{n}\cong
r_{0}\cdot ({\Z} /p^{n}{\Z}) \oplus (\lambda - r_0) \cdot ({\Z}
/p^{n-1}{\Z}).
$$
Here
$$
r_{0}=\dim_{{\F}_{p}}(S_{0})_{(p)}=\dim_{{\F}_{p}}(S_{0}/S^{p}_{0}),\  \lambda=\sum \lambda_i .
$$
Notice that $r_0$ also coincides with the number of Bernoulli numbers among
$B_{2}$, $B_{4}, \ldots, B_{p-3}$ which are divisible by $p$. The Iwasawa
invariant $\lambda$ can be defined as follows. It is well-known due to
Iwasawa and Washington (see  \cite{W}) that there exist two numbers $\lambda$
and $\nu$ called Iwasawa invariants such that $S_{n}$ has $p^{\lambda
n+\nu}$ elements for sufficiently large $n$.

Ullom's proof is based on certain assumptions about the Iwasawa number $\lambda$.
More exactly,
$$
G_{0}={\rm Gal}(F_{0}/\mathbb{Q}) \cong {\Z}/(p-1){\Z}
$$
acts on $S_{n}$ and
$$S_{n}=\bigoplus_{i=0}^{p-2} S_{n, i},$$
where $S_{n, i}=\varepsilon_{i} S_{n}$ and $\varepsilon_{i}$ are idempotents
in the group ring ${\Z}_{p}[G_{0}]$. Since we work with semi-regular $p$,
$$
\varepsilon_{i}S_{0}\cong {\Z}_{p}/B_{1,\omega^{-i}} {\Z}_{p} \ \ \ \hbox
{for} \ \ \ i=3,5,\ldots,p-2.
$$

Here $B_{1,\omega^{-i}}$ are generalized Bernoulli numbers and $\omega$ is
the Teichm\"{u}ller character of ${\Z}/(p-1){\Z}$ (see \cite{W}).

Furthermore, for each $i$ there exist $\lambda_{i}$ and $\nu_{i}$ such that
$S_{n, i}$ contains $p^{\lambda_{i}n+\nu_{i}}$ elements. Ullom's assumption
was that $\lambda_{i}<p-1$ and he conjectured that it was true for any $p$.
In this paper we will  prove that $\lambda_{i}\leq
p-1$ under assumption that $B_{1,\omega^{-i}}$ is not divisible by $p^2$.
Then we will prove that $\lambda_{i}=1$ if $B_{1,\omega^{-i}}$ is  divisible by $p^2$  that
provides almost a complete proof of  Ullom's inequality under the assumption that Vandiver's conjecture is true.
\begin{rem}
If $G_0$ acts on an abelian $p$-group $X,$ then
$X=\bigoplus_{i=0}^{p-2} X_{ i}$
with $X_{ i}=\varepsilon_{i} X.$
\end{rem}
%Furthermore, it is well-known that
%$\dim_{{\F}_{p}}(S_{0}/S^{p}_{0})=r_{0}$ and
%$\dim_{{\F}_{p}}(S_{n}/S^{p}_{n})\leq\lambda$. We will prove that
%$$\dim_{{\F}_{p}}(S_{n}/S^{p}_{n})=\lambda \ \ \ \hbox {for} \ \ \ n\geq1. $$
In our paper we will need the following presentation of $S_{n,i}$ (see \cite{W} for details).
Let $\omega$ be the Teichm\"uller character and $P_n (T)=(T+1)^{p^n} -1$.
Let $f_i (T)\in \Z_p [[T]]$ be defined by the relation $f_i ((1+p)^s -1)=L_p (s,\omega^{1-i}) , $
where $L_p$ is the $p$-adic L-function.  In this terms the Iwasawa number $\lambda_i$ is the
first coefficient of $f_i (T)$, which is not divisible by $p$.

By the $p$-adic Weierstrass Preparation Theorem $f_i(T)=U(T)p_i (T)$, where
$U(T)$ is an invertible element of $\Z_p [[T]]$ such that $U(0)=1$ and $p_i(T)$ is a unique polynomial of
degree $\lambda_i$ with the leading coefficient co-prime to $p$ and all other coefficients divisible
by $p$. Then
$$S_{n,i}=\Z_p[[T]]/(f_i,P_n)=\Z_p [T]/(p_i(T),P_n(T)).$$

\section{Second presentation of $\mathcal{V}_{n}$ and norm maps}

The following lemma was proved in \cite{S1}.

\begin{lemma}
%\noindent {\bf Lemma}
Let $A_{n}={\Z}[x]/(\frac{x^{p^{n}}-1}{x-1})$. Then $\Pic {\Z}C_{n}\cong
\Pic A_{n}$
\end{lemma}

From now on we will study $A_{n}$ instead of ${\Z}C_{n}$. Clearly,
we have the following fibre product:
\begin{equation}\label{**}
\xymatrix{A_{n+1} \ar[r]^{i_{2}}
\ar[d]^{i_{1}} &
{\Z}[\zeta_{n}] \ar[d]^{j_{2}} \\
A_{n} \ar[r]^{\hskip -1cm j_{1}} &
\frac{{\F}_{p}[x]}{(x-1)^{p^{n}-1}}:={R'_{n}}}
\end{equation}

\begin{lemma}\label{unit}
%{\bf Lemma}
$\Coker \{ j_{2}:{\Z}[\zeta_{n}] \rightarrow U({\F}_{p}[x]/ (x-1)^{p^{n}-1})
\} \cong \mathcal{V}_{n}$.
\end{lemma}

\begin{proof}
We have to prove that
\begin{gather*} \Coker (U(\mathbb{Z}[\zeta_n]) \to
U(\mathbb{F}_p [x] / (x-1)^{p^n})) =\\ \Coker (U(\mathbb{Z}[\zeta_n]) \to
U(\mathbb{F}_p [x] / (x-1)^{p^n -1})).
\end{gather*}

Clearly, it is sufficient to prove that the element $$1+(x-1)^{p^n -1} \in
U(\mathbb{F}_p [x] / (x-1)^{p^n})$$ is the image of some unit of
$\mathbb{Z}[\zeta_n]$. It is easy to see that the image of the unit $ \left(\frac{\zeta_n^{p^n
+1} - 1}{\zeta_n - 1}\right)$ under the map $\Z [\zeta_n]\to \F_p [x] / (x-1)^{p^n}$, $\zeta_n\to x$
is exactly $  1+(x-1)^{p^n -1},$ and the proof is
complete.
\end{proof}

\begin{rem}
This lemma justifies an abuse of notation $j_{1}, j_{2}, i_{1}, i_{2},
R'_{n}$ in (\ref{**}).
\end{rem}

The map $N_{n}:{\Z}[\zeta_{n}]\rightarrow A_{n}$ such that
$N_{n}(ab)=N_{n}(a)N_{n}(b)$ and the diagram below is commutative has been
introduced in \cite{S2}:
\begin{equation}\label{***}
\xymatrix{A_{n+1} \ar[r]^{i_{2}}
\ar[d]^{i_{1}} &
{\Z}[\zeta_{n}] \ar[d]^{j_{2}} \ar[ld]_{N_{n}} \\
A_{n} \ar[r]^{j_{1}}  &
R_{n}}
\end{equation}
We would like to remind the reader this construction. The following fibre product
diagram can be used for the construction without lost of generality:
$$
\xymatrix{ {\Z}_{p}[x]/(\frac{x^{p^{n+1}}-1}{x-1}) \ar[r]^{\hskip
1cm i_{2}} \ar[d]^{i_{1}} & {\Z_p}[\zeta_{n}] \ar[d]^{j_{2}} \\
{\Z}_{p}[x]/(\frac{x^{p^{n}}-1}{x-1}) \ar[r]^{\hskip 1cm j_{1}} & R_{n} }
$$
We construct $N_{n}$ using induction. If $n=1$, then
${\Z}_{p}[x]/(\frac{x^{p^{n}}-1}{x-1})\cong {\Z}[\zeta_{0}]$ and $N_{1}$ is
the usual norm map.

Commutativity of (\ref{***}) was proved in \cite{S1}. The formula
$$
\varphi_{1}(a_{1})=(a_{1}, N_{1}(a_{1}))\in
{\Z}_{p}[x]/(\frac{x^{p^{2}}-1}{x-1})
$$
defines an injective homomorphism $\varphi_{1}:U({\Z}[\zeta_{1}])\rightarrow
U({\Z}_{p}[x]/(\frac{x^{p^{2}}-1}{x-1}))$. Now we can define
$N_{2}(a_{2})=\varphi_{1}(\Norm_{{F_{2}}/F_{1}}(a_{2}))$.

Simultaneously, $N_{2}$ defines
$$
\varphi_{2}:U({\Z}_{p}[\zeta_{2}])\rightarrow
U({\Z}_{p}[x]/(\frac{x^{p^{3}}-1}{x-1}))
$$
via $\varphi_{2}(a_{2})=(a_{2}, N_{2}(a_{2}))\in
{\Z}_{p}[x]/(\frac{x^{p^{3}}-1}{x-1})$, and so on.

Proofs that all of the maps $\varphi_{i}, N_{i}$ are well-defined can be
found in \cite{S2}. They use rings
$A_{n,k}={\Z}[x]/(\frac{x^{p^{n+k}}-1}{x^{k}-1})$.

\begin{prop}
Formula $\varphi_{n-1}(a_{n-1})=(a_{n-1}, N_{n-1}(a_{n-1}))$ defines an
embedding $E_{n-1}\rightarrow U({\Z}[x]/(\frac{x^{p^{n}}-1}{x-1}))$, and
$\Coker \{j_{1}:E_{n-1}\rightarrow U(R_{n})\}\cong \mathcal{V}_{n}$.
\end{prop}

\begin{proof} Since we deal with semi-regular primes, the fact we need follows
from that of $\Norm_{{F_{n}}/F_{n-1}}(E_{n})=E_{n-1}$ and thus,
$j_{2}(E_{n})=j_{1}(E_{n-1})$ in $U(R_{n})$.
\end{proof}

Let us denote by $U_{n,k}$ the subgroup of $U({\Z}_{p}[\zeta_{n}]):=U_n ,$ which
consists of units congruent to 1 modulo $(\zeta_{n}-1)^{k}=\mu^{k}_{n}$.
% and
%$U_{n}:=U({\Z_p}[\zeta_{n}])$.

\begin{thm}
We have
$$
\mathcal{V}_{n}\cong U_{n}/(U_{n,p^{n}-1}\cdot E_{n})\cong
U_{n-1}/(U_{n-1,p^{n}-1}\cdot E_{n-1})
$$
\end{thm}

\begin{rem}
We remind the reader that $\mathcal{V}_{n}\cong U_{n}/(U_{n,p^{n}}\cdot
E_{n})$ by definition.
\end{rem}

\begin{proof} The first isomorphism is clear. Let us prove that
$\mathcal{V}_{n}\cong U_{n-1}/(U_{n-1,p^{n}-1}\cdot E_{n-1})$. The formula
$\varphi_{n-1}(a)=(a,N_{n-1}(a))$ defines an embedding
$\varphi_{n-1}:U_{n-1}\rightarrow U({\Z}_{p}[x]/(\frac{x^{p^{n}}-1}{x-1}))$.

It is sufficient to prove that the composition map $\varphi_{n-1}\cdot
j_{1}$ has the kernel $U_{n-1, p^{n}-1}$. To do this, first we note that
$U(R_{n})$ and $U_{n-1}/U_{n-1, p^{n}-1}$ have the same number of elements.
Therefore, it is enough to prove that $U_{n-1, p^{n}-1}$ is contained in the
kernel. This was proved in \cite{S2}. We would like to demonstrate the case
$n=2$.
For this, we should prove that

$(a, \Norm_{{F_{1}}/F_{0}}(a))\equiv
(1,1)\ {\rm mod}( p)$ in ${\Z}_{p}[x]/(\frac{x^{p^{2}}-1}{x-1})$ \newline
if $a\equiv 1$ mod
$\mu^{p^{2}-1}_{1}$. It is easy to see that $(a,
\Norm_{{F_{1}}/F_{0}}(a))\equiv (1,1)$ mod $p$ is equivalent to that of
$\Norm_{{F_{1}}/F_{0}}(\frac{a-1}{p})\equiv
\frac{\Norm_{{F_{1}}/F_{0}}(a)-1}{p}$ mod $p$ in ${\Z}_{p}[\zeta_{0}]$.
Since $a\equiv 1$ mod $\mu^{p^{2}-1}_{1}$, both sides are congruent to $0$
modulo $p$. The general case was proved in \cite{S2} using the rings
$A_{n,k}$ and induction in $n,k$.
\end{proof}
\begin{rem}
In fact, it is not difficult to prove that $(a, \Norm_{{F_{1}}/F_{0}}(a))\equiv
(1,1)\ {\rm mod}( p)$ in ${\Z}_{p}[x]/(\frac{x^{p^{2}}-1}{x-1})$
{\bf iff} $a\equiv 1$ mod
$\mu^{p^{2}-1}_{1}$.
\end{rem}
In the sequel we will need the following
\begin{cor}\label{N2}
Suppose $a\in U_2$ is such that $a\equiv 1 {\rm mod}\ \mu_2^{p^2 -1}.$
Then

$\Norm_{{F_{2}}/F_{1}}(a)\equiv 1 {\rm mod}\ \mu_1^{p^2 -1}.$
\end{cor}
\begin{proof}
Consider the diagram \ref{***} for $n=2$. Then\newline
$N_2 (a)=(\Norm_{{F_{2}}/F_{1}}(a),\Norm_{{F_{2}}/F_{0}}(a))\equiv (1,1) {\rm mod} (p)$ in
${\Z}_{p}[x]/(\frac{x^{p^{2}}-1}{x-1})$. \newline
Consequently,
$\Norm_{{F_{2}}/F_{1}}(a)\equiv 1\ {\rm mod}\ \mu_1^{p^2 -1}$.
\end{proof}
\section{Number of elements in $\mathcal{V}_{n}^{+}$ }

Let us introduce integers $r_{n}$ as the number of elements in $E_{n,
p^{n+1}-1}/E^{p}_{n, p^{n}+1}$. Similarly, let $r_{n,i}$ be the number of
elements in $\varepsilon_i (E_{n,
p^{n+1}-1}/E^{p}_{n, p^{n}+1} )$. In particular, it follows that $r_{n}=\sum r_{n,i}$,
$r_{0,i}=1$ if $\lambda_{p-i}>0$, otherwise $r_{0,i}=r_{k,i}=0$.

\begin{lemma}
If $\epsilon\in E_{n, p^{n}+1}$, then $\epsilon$ is real and therefore,
$E_{n, p^{n}+1}=E^{+}_{n, p^{n}+1}$.
\end{lemma}

\begin{thm}\label{thm}
Let $\alpha$ be an ideal of ${\Z}[\zeta_{n}]$ such that $\alpha^{p}=(q)$.
Let $q\equiv 1$ mod $\mu^{p^{n+1}-1}_{n}$. Then $q\equiv 1$ mod
$\mu^{p^{n+1}}_{n}$.
\end{thm}

Before we give a proof of the theorem, let us formulate its consequence,
which we will need in sequel.

\begin{cor}
 $E_{n, p^{n+1}-1}=E_{n, p^{n+1}+1}$.
\end{cor}

\noindent{\it Proof of Theorem \ref{thm}.} Consider the extension
$F_{n}(\sqrt[p]{q})/F_{n}$. Only $\mu_{n}$ ramifies in this extension. Let
$\epsilon\in E_{n}$. Then for any valuation $v \neq \mu, \ \epsilon$
is a norm in the corresponding extension of local fields
$F_{n,v}(\sqrt[p]{q})/F_{n,v}$. Therefore, the local norm residue symbol
with values in the group of $p$-th roots of unity $(\epsilon, q)_{v}=1$.
By the product formula, $(\epsilon, q)_{\mu_{n}}=1$. Set
$\epsilon=\zeta_{n}$. If $q\equiv 1$ mod $\mu^{p^{n-1}}_{n}$ but $q\neq
1$ mod $\mu^{p^{n}}_{n}$, then simple local computations (see for instance
\cite{CF}) show that $(\zeta_{n}, q)_{\mu_{n}}\neq 1$. The theorem is
proved.\qed

\begin{thm}
The number of elements in $\mathcal{V}^{+}_{n}$ is
$p^{r_{0}+\ldots+r_{n-1}}$.
\end{thm}

\begin{proof}
If $n=1$, then it was proved in \cite{KM}. Let us denote the number of
elements in group $A$ by $|A|$. Assume that
$|\mathcal{V}^{+}_{n}|=p^{r_{0}+\ldots+r_{n-1}}$. Let us prove that
$|\mathcal{V}^{+}_{n+1}|=p^{r_{0}+\ldots+r_{n-1}+r_{n}}$. Indeed,
$|(U_{n}/(U_{n, p^{n}}\cdot E))^{+}|=p^{r_{0}+\ldots+r_{n-1}}$. Clearly,
$(U_{n}/(U_{n, p^{n}}\cdot E))^{+}= U^{+}_{n}/(U^{+}_{n, p^{n}}\cdot E^{+})$
and $U^{+}_{n, p^{n}}=U^{+}_{n, p^{n}+1}$ since $p$ is odd. Taking into
account that $\mathcal{V}^{+}_{n+1}\cong U^{+}_{n}/U^{+}_{n, p^{n+1}-1}\cdot
E^{+}_{n}$, it remains to prove that
$$
\left| \frac{U^{+}_{n, p^{n}+1}\cdot E^{+}_{n}}{U^{+}_{n, p^{n+1}-1}\cdot
E^{+}_{n}} \right| =p^{r_{n}}.
$$
Let us use the isomorphism
$$
\frac{U^{+}_{n, k}\cdot E^{+}_{n}}{E^{+}_{n}}\cong U^{+}_{n, k}\cdot E^{+}_{n, k},
$$
which shows that we have to prove that
$$
\left| \frac{U^{+}_{n, p^{n}+1}}{U^{+}_{n, p^{n+1}-1}} \right| : \left|
\frac{E^{+}_{n, p^{n}+1}}{E^{+}_{n, p^{n+1}+1}}\right| =p^{r_{n}}.
$$
It is easy to see that
$$
\left|
\frac{U^{+}_{n, p^{n}+1}}{U^{+}_{n, p^{n+1}+1}}
\right|
=p^{\frac{p^{n+1}-p^{n}}{2}-1}.
$$
The second number can be computed as follows:
$$
\left| \frac{E^{+}_{n, p^{n}+1}}{E^{+}_{n, p^{n+1}+1}} \right|= \left
|\frac{E_{n, p^{n}+1}}{(E_{n, p^{n}+1})^{p}} \right|: \left|\frac{E_{n,
p^{n+1}+1}}{(E_{n, p^{n}+1})^{p}}
\right|=p^{\frac{p^{n+1}-p^{n}}{2}-1}:p^{r_{n}}
$$
and the theorem is proved.
\end{proof}

Closing this section we would like to mention the following

\begin{prop}
$r_{0}\leq r_{1}\leq \ldots\leq\lambda=\sum \lambda_i$.
\end{prop}

\begin{proof}
Let $\epsilon \in E_{n, p^{n+1}+1}/(E_{n, p^{n}+1})^{p}$. Then the
extension $F_{n}(\sqrt[p]{\epsilon})/F_{n}$ is unramified, which defines
an embedding $E_{n, p^{n+1}+1}/(E_{n, p^{n}+1})^{p}$ into $S^{\ast}_{n}$. It
is easy to see that the canonical embedding $S^{\ast}_{n}\rightarrow
S^{\ast}_{n+1}$ defines an embedding
$$
E_{n, p^{n+1}+1}/(E_{n, p^{n}+1})^{p}\rightarrow E_{n+1, p^{n+2}+1}/(E_{n, p^{n+1}+1})^{p}.
$$
Therefore, $r_{n}\leq r_{n+1}$.

Furthermore, because of the projection $S^{\ast}_{n}\rightarrow
\mathcal{V}^{+}_{n+1}$ (see \cite{KM}) it is clear that
$p^{\lambda {n}+\nu}\geq p^{r_{0}+\ldots+r_{n}}$, and the latter inequality
implies that $r_{n}\leq \lambda$.
\end{proof}
%XXX
\begin{cor}
If $p$ divides $B_{1,\omega^{-i}}$, then the number of elements in $\varepsilon_i (\mathcal{V}^{+}_{n})$ is
$p^{1+r_{1,i}+\ldots+r_{n-1,i}}$ and $1\leq r_{1,i}\leq \ldots \leq r_{k,i}\leq \lambda_{p-i} .$
\end{cor}
%In the next paper we will show that for semi-regular primes
%$r_{1}=r_{2}=\ldots=\lambda$.

\section{Weak Kervaire-Murthy Conjecture and New Link between $S$ and $\mathcal{V}$ Groups}

In this section let us denote by $(a,b)$ the local norm residue symbol with
values in $p$-th roots of unity. Here $(a,b)$ are elements of the completion
of $F_{n}$ with respect to $\mu_{n}$. Assume that $a\in U_{n,k}\setminus
U_{n, k+1}$, $b\in U_{n, p^{n+1}-k}\setminus U_{n, p^{n+1}-k+1}$, and $k$ is
prime to $p$.

\begin{lemma}[see \cite{CF}]
$(a,b)\neq 1$.
\end{lemma}

%\begin{proof}
%See \cite{CF}.
%\end{proof}

\begin{thm}
Let $\alpha \in S^{(p)}_{n}$ and $\alpha^{p}=(q)$. Then the formula
$f_{\alpha}(x)=(x,q)$, $x\in \mathcal{V}^{+}_{n+1}$ defines a non-trivial
character of $\mathcal{V}^{+}_{n+1}$ (if $\alpha$ is not trivial).
\end{thm}

\begin{proof}

{\bf Step 1.}
If $q\equiv 1$ mod $\mu^{p^{n+1}-1}_{n}$, then $\alpha = 1\in S_{n}$.

Indeed, we already know that $q\equiv 1$ mod $\mu^{p^{n+1}}_{n}$
and hence the extension $F_{n}(\sqrt[p]{q})/F_{n}$ is non-ramified.
Therefore, $q=\varepsilon \cdot a^{p}$ for some $\varepsilon \in E_{n}, \ a\in F_{n}$
and consequently $\alpha =1$ in $S_{n}$.

\medskip

{\bf Step 2.} Without lost of generality we can assume that $q\in
U_{n,k}\setminus U_{n, k+1}$ with $k<p^{n+1}-1$ and $k$ being prime to $p$.

Indeed, if $k=p\cdot s$, then $q=1+a_{0}\mu^{ps}_{n}+t\mu^{ps+1}_{n}$, where
$a_{0}$ is an integer prime to $p$. Easy computations show that
$q(1-a_{0}\mu^{s}_{n})^{p}\in U_{n, k+1}$. Proceeding in this way, we can
find $q_{1}\in U_{n, k_{1}}$ such that $(q_{1})=(\Gamma \alpha)^{p}, \
\Gamma\in U(F_{n})$, and such that $k_{1}$ is prime to $p$.
\medskip

{\bf Step 3.} $1+\mu^{p^{n+1}-k}_{n}\in \mathcal{V}_{n+1}$.

Indeed, if $1+\mu^{p^{n+1}-k}_{n} \equiv \varepsilon$ mod
$\mu^{p^{n+1}-k}_{n}$, $\varepsilon\in E_{n}$, then $(\varepsilon, q)=1$.
However, it is not true by Step 2 and Lemma of this section.
\medskip

{\bf Step 4.} Since $S_{n}=S^{-}_{n}$, the character constructed above is a
non-trivial character of the group $\mathcal{V}^{+}_{n+1}$. The proof is
complete.
\end{proof}

\begin{cor}{("weak Kervaire--Murthy conjecture")}\label{KMC}
$S^{(p)}_{n}\cong
(\mathcal{V}^{+}_{n+1}/(\mathcal{V}^{+}_{n+1})^{p})^{\ast}$.
\end{cor}

\begin{cor}
$S^{\ast}_{n}/(S^{\ast}_{n})^{p}\cong
\mathcal{V}^{+}_{n+1}/(\mathcal{V}^{+}_{n+1})^{p}$.
\end{cor}

\begin{proof}
This follows from the existence of the surjection $S^{\ast}_{n}\rightarrow
\mathcal{V}^{+}_{n+1}$ constructed in \cite{KM}.
\end{proof}
\begin{thm}
There exists a
canonical embedding $$i: S_{n-1}^{(p)}\to \mathcal{V}_n^{-} /(\mathcal{V}_n^-)^p $$
\end{thm}
\begin{proof}
Let $\alpha $ be an ideal such that $\alpha^p =(q)$. Define $i(\alpha )=q$. This map is well-defined because
the number $q$ is defined up to a transformation $q\to \epsilon r^p q$, where $\epsilon \in E_n$.
Clearly, the images of $q$ and  $\epsilon r^p q$ coincide in $\mathcal{V}_n /(\mathcal{V}_n)^p $.
If $\alpha\in Ker (i)$, then $q\equiv \epsilon r^p$ mod $(1-\zeta_n )^{p^n -1}$ and it follows
from the {\bf Step 1} of the proof of the previous theorem that $\alpha =1$. Hence, $i$ is an embedding.

Since $S = S^{-}$, it follows that $i$ maps $S_{n-1}$ into
$(\mathcal{V}_n /(\mathcal{V}_n)^p)^- =\mathcal{V}_n^{-} /(\mathcal{V}_n^-)^p $.

\end{proof}

\section{Ullom's inequality}
The aim of this section is to prove the following result (a weaker version of Ullom's inequality):
\begin{thm}\label{UL}
Let the generalized Bernoulli number $B_{1,\omega^{-i}}$ is divisible by $p$ but not by $p^2$.
Then the corresponding Iwasawa number $\lambda_i$ is less than $p.$
\end{thm}
\begin{proof}
The proof will consist of several lemmas. For technical reasons it will be easier to deal
with the original definition of $\V_2$, namely
$$\mathcal{V}_{k}=\Coker\{j_{2}:E_{k}{\longrightarrow} U(R_{k})\},\ k=1,2;$$
see Introduction.

Let us also make an important note: $R_2 = \Z[\zeta_2]/(\zeta_2 -1)^{p^2}$ and
$R_1 = \Z[\zeta_1]/(\zeta_1 -1)^{p}$

%BBB

\begin{lemma}\label{L1}
The map $\pi : R_2\to R_1$ defined as $\pi (x)=x^p$ is an surjective homomorphism
of rings, which induces an epimorphism of the corresponding groups of units
$\pi : U(R_2)\to U(R_1)$ and $\pi : \V_2\to\V_1$.
\end{lemma}
\begin{proof}
$\pi (a+b) =(a+b)^p=\pi(a)+\pi (b)$ and $\pi (m)=m^p =m,\ m\in\Z$ because
$(p)=(\zeta_2 -1)^{p^3 -p^2}=(\zeta_1 -1)^{p^2 -p}=0\in R_2.$
Furthermore, $\pi (\zeta_2 -1) =(\zeta_2 -1)^p =\zeta_1 -1 $ and
$\pi (1+(\zeta_2 -1)^k) =(1+(\zeta_2 -1)^k)^p =1+(\zeta_1 -1)^k $
what proves that $\pi$ is surjective homomorphism of the rings and the corresponding groups of units.

To prove that $\pi : U(R_2)\to U(R_1)$ induces a surjection
$\pi : \V_2\to\V_1$ we need to prove that $\pi (E_2)\subseteq E_1 .$ Since $p$ satisfies
Vandiver's conjecture, we can use the subgroup of cyclotomic units $C(k)\subset E_k$ insted
of $E_k .$ This means that we have to show that $\pi (C(2))\subseteq C(1) .$ However, this is clear
because of our previous computations:
$$
\pi \left( \frac{\zeta_2^m -1}{\zeta_2 -1}\right) =\pi (1+\zeta_2+...+\zeta_2^{m-1})=1+\zeta_1+...+\zeta_1^{m-1}=
\frac{\zeta_1^m -1}{\zeta_1 -1}.
$$
The latter computation completes the proof.
\end{proof}
\begin{cor}\label{M1}
If $x\in\V_2$ is such that $x^p=1$, then $x\in ker(\pi).$
\end{cor}

%Canonical maps $\rm{id}: \F_p [x]\to \F_p [x]$ and ${\rm Norm}_{F_1 /F_0}: \Z [\zeta_1]\to \Z [\zeta_0]$
%induce one and the same epimorphism ${\rm Norm}_{\V_2 / \V_1}: \V_2\to \V_1 .$
%\end{lemma}
%\begin{proof} The statement follows from the following facts:
%\begin{itemize}
%\item ${\rm Norm}_{F_1 /F_0}(\zeta_1)= \zeta_0$;
%\item ${\rm Norm}_{F_1 /F_0}(a+b)\equiv {\rm Norm}_{F_1 /F_0}(a)+{\rm Norm}_{F_1 /F_0}(b)\ {\rm{mod}} (p)$,
%here $a,b\in \Z[\zeta_1]$. For proof, see \cite{S1}.
%\end{itemize}
%\end{proof}
%\begin{lemma}\label{L2}
%Maps $i: \F_p [x]\to \F_p [x]$ given by the formula $i(x)=x^p$ and the canonical embedding $i_{F_1 /F_0}: \Z [\zeta_0]\to \Z %[\zeta_1]$
%induce one and the same homomorphism $i_{\V_2 / \V_1}: \V_1\to \V_2 .$
%\end{lemma}
%\begin{proof}
%Clearly, it is sufficient to prove that the element $$1+(x-1)^{p^2 -p} \in
%U(\mathbb{F}_p [x] / (x-1)^{p^2 -1})$$ is the image of some unit of
%$\mathbb{Z}[\zeta_1]$. One can check that the unit $\left(\frac{\zeta_1^{p+1}-1}{\zeta_1 -1}\right)^p$
%is the required element (see also the proof of \ref{unit}).
%\end{proof}
%TTT
\begin{lemma}\label{L3}
If $\lambda_i\geq p,$ then $S_{1,i}$ has $p$ generators as an abelian group (we assume that $p$ is semi-regular).
\end{lemma}
\begin{proof}
We have already mentioned in Introduction that it follows from results of \cite{W} that $$S_{1,i}\cong \frac{\Z_p [T]}{((T+1)^p -1,p_i (T))} ,$$
where $p_i (T)$ is a polynomial of degree $\lambda_i$ and such that all the coefficients except of the leading
one  are divisible by $p.$ Clearly,
$$S_{1,i}/(S_{1,i})^p\cong \frac{\Z_p [T]}{(p,T^p ,T^{\lambda_i})}=\frac{\Z_p [T]}{(p,T^p)},$$
what proves the lemma.
\end{proof}
\begin{lemma}\label{S}
If $\lambda_i\geq p$ and $b_i:=B_{1,\omega^{-i}}$ is divisible by $p$ but not by $p^2$,
then $S_{1,i}\cong \V_{2,p-i}\cong (\F_p)^p$ as abelian groups. Here $\V_{2,p-i}=\varepsilon_{p-i}\V_2$.
\end{lemma}
%MMM
\begin{proof} Consider the following fibre product:
$$
\xymatrix{\Z_p[T] / ((T+1)^p -1)\ar[r]\ar[d]^{i_1} &
\Z_p [\zeta_{0}]=\Z_p [T]/(\frac{(T+1)^p -1}{T}) \ar[d]^{j_1}\\
\Z_p [T]/(T)=\Z_p \ar[r]  &
\F_p}.
$$
Here, $i_1 (T)=0$, $j_1 (\zeta_0)=1$
and horizontal maps are defined by $T\to T$, $1\to 1$.

Let us write elements of $\Z_p[T] / ((T+1)^p -1)$ as pairs $(x\in \Z_p,y\in \Z_p [\zeta_0])$
with clear compatibility conditions.
In order to prove that $S_{1,i}\cong (\F_p)^p$, it is sufficient to show that $(p_i (0), p_i (\zeta_0 -1)$
divides $(p,p)$ in $\Z_p [T]/((T+1)^p -1)$.

Indeed, $p_i (0)=b_i$ and $p_i (\zeta_0 -1)=b_i + \sum a_k (\zeta_0 -1)^k +(\zeta_0 -1)^{\lambda_i}$.
Since $p$ divides $a_k$, $(p)=(b_i)=(\zeta_0 -1)^{p-1}$, and ${\lambda_i}\geq p$, we see that $p_i (\zeta_0 -1)=b_i(1+(\zeta_0 -1)X$
and therefore, $(p_i(0), p_i (\zeta_0 -1))=(b_i,b_i)\times (1,1+(\zeta_0-1)X)$. It follows that $(p_i(0), p_i (\zeta_0 -1))$ divides
$(p,p)$ and due to the weak Kervaire-Murthy conjecture the lemma is completely proved.
\end{proof}

Now, we can finish proof of the theorem.
\vskip0.3cm
Due to \ref{L1} we have a surjection $\pi: \varepsilon_{p-i} \V_2 \to \varepsilon_{p-i} \V_1$.
On the other hand, due to \ref{M1} and \ref{S}  $\pi(\varepsilon_{p-i}\V_2)=1$.
This contradiction completes the proof of the theorem.

%\begin{itemize}
%\item We have proved that under the assumptions of the theorem, $S_{1,i}^{\ast}\cong \V_{2,p-i}^+$. The duality is given
%by the local norm residue symbol in $\Z_p [\zeta_1]$, see Section 4. More exactly, if $\alpha\in S_{1,i},\ v\in \V_{2,p-i}^+$ %and $\alpha^p=(q)$,
%then the duality is defined by $<\alpha, v>=(q,v)$.
%\item Let $\sigma\in {\rm Gal(F_1 /F_0)}$.
%Then $(\sigma (q),v)=(q,\sigma^{-1}(v))^\sigma=(q,\sigma^{-1}(v))$ because $ (q,\sigma^{-1}(v))$ is a $p$-root of unity in our %case.
%\item Therefore, $({\rm{Norm}}_{F_1/F_0} (q),v)=(q,{\rm{Norm}}_{\V_2/\V_1} (v))$. Here, both norms can be considered as %elements of
%$\Z [\zeta_1]$ or $\Z_p [\zeta_1]$ due to \ref{L1} and \ref{L2}.
%\item Furthermore, \ref{L1} and \ref{L2} imply that $i({\rm{Norm}}_{\V_2/\V_1}(v))=v^p$ in
%$\V_{2,p-i}^+ .$
%\item Since $\V_{2,p-i}^+ $ has exponent $p,$ it follows that
%$$({\rm{Norm}}_{F_1/F_0} (q),v)=(q,{\rm{Norm}}_{\V_2/\V_1} (v))=(q,v^p)=1$$
%for any $\alpha\in S_{1,i},\ \alpha^p=(q)$ and $v\in \V_{2,p-i}^+$.
%\item Therefore, for any $\alpha\in S_{1,i}$ and any $v\in\V_{2,p-i}^+$, we have: \newline
%$<{\rm{Norm}}_{F_1/F_0} (\alpha), v>=1$.
%\item It follows that for any $\alpha\in S_{1,i}$, ${\rm{Norm}}_{F_1/F_0} (\alpha)$ is a principal ideal of $S_{1,i}$.
%However, the norm map ${\rm{Norm}}_{F_1/F_0}: S_{1,i}\to S_{0,i} $ is surjectve and the canonical map $S_{0,i}\to S_{1,i} $
%is injective (see \cite{KM} and references therein).
%\item This contradiction completes the proof of the theorem.
%\end{itemize}
\end{proof}
\section{Further relations between Bernoulli and Iwasawa numbers}
The aim of this section is to prove that if the generalized
Bernoulli number $b_i=B_{1,\omega^{-i}}$ is divisible  by $p^2$, then
the Iwasawa number $\lambda_i =1$

\subsection{Fine structure of  $\V_{2,p-i}^+$ if $p^2$ divides $b_i$}
\begin{thm}
Let $b_i=p^{k_i}t,\ k_i\geq 2,$ where $t $ is co-prime to $p$.
Then
%$S_{1,i}\cong \F_p^{k_i +1}$ and
$ (\V_{2,p-i})^+\cong (\Z/(p^2) )\oplus \F_p^k$, where $k=min(\lambda_i -1, p-1)$.
\end{thm}
\begin{proof}
$\V_{2,p-i}^+$ is a factor  of
$\varepsilon_{p-i}(V_{2})=\varepsilon_{p-i}(U_1 /U_{1, p^2 -1})=\varepsilon_{p-i}(U_1 /U_{1, p^2 +1})$
because $p-i$ is an even number between $2$ and $p-3.$
It is easy to prove that $\varepsilon_{p-i}(V_{2})=(\Z/(p^2))\oplus \F^{p-1}$.

It follows from the weak Kervaire--Murthy conjecture and \ref{S} that
$\V_{2,p-i}^+\cong (\Z/(p^2) )\oplus \F_p^k$ or $\V_{2,p-i}^+\cong  \F_p^{k+1}.$
So, we have to exclude the second possibility.

Let us denote the local norm residue symbol with values in $p$-roots of unity from the section 4
by $(a,b)_{n,0}$. In particular, we are interested in
$(a,b)_{1,0}$ and $(c,d)_{0,0}$. Let us also consider the local norm residue symbol with values in $p^2$-roots of unity, which we denote by $(a,b)_{n,1}$. Note that it is defined if $n>0.$

Let us make the following easy remarks. To simplify notations, from now on we denote
$\V_{n,p-i}^+$ by $\V_{n,p-i}$
\begin{itemize}
\item $(a^p,b)_{1,1}=(a,b)_{1,0}$;
\item if $b\in \Z_p [\zeta_0]$, then $(a,b)_{1,0}=(\rm{Norm}_{F_1 /F_0}(a),b)_{0,0}$;
\item $(1+(\zeta_0 -1)^i, 1+(\zeta_0 -1)^{j})_{0,0}=1,$ if $i+j>p$;
\item $(1+(\zeta_0 -1)^i, 1+(\zeta_0 -1)^{p-i})_{0,0}\neq 1.$
\end{itemize}
Let $v_2$ be the image of $ 1+(\zeta_1 -1)^{p-i}$ in $\V_{2,p-i}$.
Let $v_1$ be the image of $ 1+(\zeta_0 -1)^{p-i}$ in $\V_{1,p-i}\cong \Z/(p)$.
Clearly, ${\rm{Norm}}_{F_1/F_0}(v_2)$ generates the same element in $\V_{1,p-i}\cong \Z/(p)$
as $v_1 .$ We will write ${\rm{Norm}}_{F_1/F_0}(v_2)=v_1 .$

%\subsection{Duality between $\V_{2,p-i}$ and $S_{1,i}^{(p^2)}$}

\begin{lemma}
Let us assume that $p^2$ divides $b_i$.
Then there exists an  ideal $\alpha\subset \Z[\zeta_0]$, whose class belongs to
$S_{0,i}$ such that
$\alpha^{p^2}=(q),\ q\in \Z[\zeta_0]$ and $\alpha^{p}$ is not a principal ideal.
\end{lemma}
\begin{proof}
The statement follows from the fact that $S_{0,i}\cong \Z_p /(b_i).$
\end{proof}
\begin{lemma}
Let $v_2, q$ be as above. Then $(v_2, q)_{1,1}$ is a primitive $p^2$-root of unity.
\end{lemma}
\begin{proof}

Since $\alpha$ generates an element of $S_{0,i}$, it follows from results of Section 4 that
$q$ can be chosen such that $q\equiv 1 {\rm{mod}} (\zeta_0 -1)^i.$
Let us compute $(v_2, q)_{1,1}^p$.
$$
(v_2, q)_{1,1}^p=(v_2^p, q)_{1,1}=(v_2, q)_{1,0}=({\rm{Norm}}_{F_1/F_0}(v_2),q)_{0,0}=(v_1,q)_{0,0}\neq 1
$$
Since $(v_2, q)_{1,1}^p$ is a non-trivial $p$-root of unity, clearly
$(v_2, q)_{1,1}$ is a $p^2$-root of unity.
\end{proof}
\begin{lemma}
Let $q$ be as above. Then the formula $\ <v,\alpha>=(v, q)_{1,1}$ defines a character
of
$\V_{2,p-i}$.

\end{lemma}
\begin{proof}
We have to prove that $(v, q)_{1,1}=1$ if $v$ is a unit of $\Z[\zeta_1]$ or\newline
$v\equiv 1\ \rm{mod} (1-\zeta_1)^{p^2+1}$, the latter because\newline
$\varepsilon_{p-i}(V_{2})=\varepsilon_{p-i}(U_1 /U_{1, p^2 -1})=\varepsilon_{p-i}(U_1 /U_{1, p^2 +1})$.

If $v$ is a unit, then the extension $F_1 (v^{1/p^2})/F_1$ can ramify at $(1-\zeta_1)$ only.
Furthermore, for any prime $\theta\neq (1-\zeta_1)$ we have $q=r^{p^2}\times \it{local\ unit}$
for some $r\in (F_0)_{\theta}$. It follows that $(v,q)_{\theta} =1$, here $(v,q)_{\theta}$ is the
corresponding local symbol with values in $p^2 $-roots of unity. The product formula implies that
$(v, q)_{1,1}=1$.

It remains to prove that  $(v, q)_{1,1}=1$ if $v\equiv 1\ \rm{mod} (1-\zeta_1)^{p^2+1}$.
Indeed, $v=t^p$ for some $t\in\Z_p [\zeta_1]$ such that $t\equiv 1 \ \rm{mod} (1-\zeta_1)^{p+1}$  and
$(v, q)_{1,1}=(t^p,q)_{1,1}=(t,q)_{1,0}=({\rm{Norm}}_{F_1/F_0}(t),q)_{0,0}=1$ because
${\rm{Norm}}_{F_1/F_0}(t)\equiv 1 \ \rm{mod} (p)$ and $q$ can be chosen to satisfy
$q\equiv 1 \mod (1-\zeta_0)^2 $.
\end{proof}
Now we can finish the proof of the theorem. We have proved that $(v,q)_{1,1}$ is a character of
$\V_{2,p-i}$ and since $(v_2, q)_{1,1}$ is a primitive $p^2$-root of unity, we can exclude the possibility
$\V_{2,p-i}\cong \F_p^k.$
\end{proof}

\subsection{The Main Theorem I}

\begin{thm}\label{M}
Assume $p^2$ divides $b_i$, $q$ is the same as in the previous subsection, and $\sigma$ is a generator of $\rm{Gal}(F_1/F_0)$ such that
$\sigma(\zeta_1)=\zeta_1^{p+1}$.
Then $\sigma (v_2)/v_2 =v_2^p$, where $v_2$  is a generator of $\V_{2,p-i}$ such that
$(v_2, q)_{1,1}=\zeta_1$
\end{thm}
\begin{proof}
Let us consider $(\sigma(v_2), q)_{1,1}$, where $q$ is the same as in the previous subsection.
We have $(\sigma(v_2), q)_{1,1}=\sigma((v_2, q)_{1,1})=\zeta_1^{p+1}=(v_2, q)_{1,1}^{p+1}$.
Hence, $(\sigma(v_2)/v_2, q)_{1,1}=(v_2^p,q)_{1,1}=\zeta_0 .$

Let us consider the annihilator of $\sigma (v_2)/v_2$ in the character group $(\V_{2,p-i})^{\ast}$.
We denote it by $Ann (\sigma (v_2)/v_2) $.
\begin{lemma}
$Ann (\sigma (v_2)/v_2) \cong \F_p^{k+1},$ where
$k=min(\lambda_i -1, p-1).$
\end{lemma}
\begin{proof}
The proof consists of three statements below.

\begin{itemize}
\item $\sigma (v_2)/v_2$ has order $p$.
Indeed,
$(\sigma(v_2)/v_2, q)_{1,1}=\zeta_0 .$ Since $ (\V_{2,p-i})^{\ast}\cong (\Z/(p^2) )\oplus \F_p^k$
and $q$ generates a character of order $p^2$, it follows that the value of any character on
$\sigma (v_2)/v_2$ is either a primitive $p$-root of unity or $1.$ Consequently, $\sigma (v_2)/v_2$
has order $p$ (it cannot be $1$ because  $(\sigma(v_2)/v_2, q)_{1,1}=\zeta_0 $).
\item Therefore, $ (\V_{2,p-i})^{\ast}/Ann (\sigma (v_2)/v_2) \cong \F_p .$
\item Since $q$ generates a character of order $p^2$, $(\sigma(v_2)/v_2, q)_{1,1}=\zeta_0 ,$
and $ (\V_{2,p-i})^{\ast}\cong (\Z/(p^2) )\oplus \F_p^k$,
we can deduce that $Ann (\sigma (v_2)/v_2) \cong \F_p^{k+1}.$
\end{itemize}

\end{proof}
Now, we can complete the proof of the theorem. By the Kervaire--Murthy theorem,
$Ann (\sigma (v_2)/v_2)$ is a subgroup of $S_{1,i}$. By the lemma above and the
weak Kervaire--Murthy conjecture, we have $Ann (\sigma (v_2)/v_2)\cong S_{1,i}^{(p)}$,
the subgroup of elements of order $p$. Thus, for any ideal $\alpha_m$ such that
$\alpha_m^p=(q_m)\subset \Z [\zeta_1]$ we have $(\sigma (v_2)/v_2, q_m)_{1,0}=1$ and
$(\sigma (v_2), q_m)_{1,0}=(v_2,q_m)_{1,0}=(v_2,q_m)_{1,0}^{p+1}.$ It follows that
$(\sigma (v_2)/v_2, q_m)_{1,0}=(v_2^p,q_m)_{1,0}$. Therefore, for any character
$\chi\in \V_{2,p-i}^{\ast}$, we have $\chi (\sigma(v_2)/v_2)=\chi (v_2^p)$ and
consequently $\sigma (v_2)/v_2=v_2^p .$

\end{proof}
\subsection{Main Theorem II}
\begin{lemma}
$\varepsilon_{p-i}(V_2)$ and $\V_{2,p-i}$ are $\Z_p [[T]]$-modules with one generator.
Here the action is defined as follows: $T\cdot v=\sigma (v)/v$ and $a\cdot v=v^a,\ a\in\Z_p$.
\end{lemma}
\begin{proof}
$\varepsilon_{p-i}(V_{2})=\varepsilon_{p-i}(U_1 /U_{1, p^2 -1})$.
Since $\varepsilon_{p-i}(U_1)$ is an $\Z_p [[T]]$-modules with one generator, it is
also true for its factors $\varepsilon_{p-i}(V_{2})$ and $\V_{2,p-i}$ because
$U_{1, p^2 -1}$ and the image of $U(\Z[\zeta_1])$ in $V_2$ are
$\Z_p [[T]]$-submodules.

\end{proof}

\begin{lemma}
$\varepsilon_{p-i}(V_2)\cong \Z_p [[T]]/(T^p,pT,p^2).$
\end{lemma}
\begin{proof}
It is easy to verify that $pT$ and $p^2$ annihilate $V_2 .$ Further,
$\varepsilon_{p-i}(U_1)$ is annihilated by $(T+1)^p -1$. Since $pT$
annihilates $\varepsilon_{p-i}(V_2)$, we deduce that $T^p$ annihilates it too.
Finally, it is easy to see that both
$\varepsilon_{p-i}(V_2)$ and $\Z_p [[T]]/(T^p,pT,p^2)$ contain $p^{p+1}$ elements.
The last observation completes the proof.
\end{proof}
\begin{thm}
If the generalized
Bernoulli number $b_i=B_{1,\omega^{-i}}$ is divisible  by $p^2$, then
the Iwasawa number $\lambda_i =1$
\end{thm}
\begin{proof}
It follows from \ref{M} that $T-p$ annihilates $\V_{2,p-i}$. Therefore,
as a $\Z_p [[T]]$-module $\V_{2,p-i}$ factors through
$\Z_p [[T]]/(T^p, pT, p^2, T-p)\cong \Z_p/(p^2)$.
Since we already know that
$\V_{2,p-i}\cong\Z/(p^2)\oplus \F^k$, where $k=min(\lambda_i -1, p-1)$, we conclude
that $\V_{2,p-i}\cong\Z/(p^2)$ and $\lambda_i =1.$
\end{proof}
\begin{cor}
$\V_{n,p-i}\cong \Z/(p^n) $  if $p^2$ divides $b_i$.
\end{cor}
\begin{proof}
It is an easy consequence of Corollary 3.6.
\end{proof}
\begin{cor}
$S_{n,i}\cong \Z/(p^{n+k_i})$ if $p^2$ divides $b_i$. Here $k_i$ is the
$p$-adic valuation of $b_i .$
\end{cor}
\begin{proof}
The statement follows from the fact that
${\rm{Norm}}_{F_{n+1}/F_n} (i_{F_{n+1}/F_n})(\alpha)=\alpha^p$ for any ideal $\alpha\subset \Z[\zeta_n] $
and that of ${\rm{Norm}}_{F_{n+1}/F_n}: S_{n+1}\to S_n$ is surjective while
$i_{F_{n+1}/F_n} :S_n\to S_{n+1}$ is injective.
\end{proof}

\section{Fine structure of $\V_{n,p-i}$ and $S_{n,i}$ if $p^2$ does not divide $b_i$}
Throughout this section we assume that the $p$-adic valuation $v_p (b_i)=1 .$
We already know that if $p^2$ does not divide $b_i$, then $\lambda_i$ satisfies
Ullom's inequality $\lambda_i\leq p-1$ and $S_{0,i}\cong\F_p .$
\subsection{Fine structure of $\V_{n,p-i}$}
\begin{lemma}
Let $\alpha\in S_{0,i}$. Then $\alpha=\beta^p ,$ where $\beta\in S_{1,i} .$
\end{lemma}
\begin{proof}
We consider $S_{1,i}$ as a $\Z_p [[T]]$-module.
It follows from results of \cite{W} that

$$
S_{1,i}\cong \frac{\Z_p [[T]]}{((T+1)^p -1,f_i (T))} =
\frac{\Z_p [T]}{((T+1)^p -1,p_i (T))},
$$
where $f_i (T), p_i(T)$ were defined in Introduction.

Clearly,
$S_{1,i}/(S_{1,i})^p\cong \frac{\Z_p [T]}{(p,T^p ,T^{\lambda_i})}=
\frac{\Z_p [T]}{(p,T^{\lambda_i})}=\F_p[T]/(T^{\lambda_i}),$
because of Ullom's inequality. Let us prove that the image of $S_{0,i}$ under the canonical
embedding $i_{F_1 /F_0}: S_{0,i}\to S_{1,i}$ is contained in $S_{1,i}^p$. Indeed, this image is generated
by $N(T)=1+(T+1)+\cdots +(T+1)^{p-1}=((T+1)^p -1)/T$. Again, because of Ullom's inequality, the image
of $N(T)$ in $S_{1,i}/(S_{1,i})^p\cong \frac{\Z_p [T]}{(p,T^{\lambda_i})}$ is zero. The lemma is proved.
\end{proof}
The crucial step in computation of $\V_{n,p-i}$ is to consider the case $n=2 .$
From the weak Kervaire--Murthy conjecture and Ullom's inequality, we know that as an abelian group
$\V_{2,p-i}$ has $\lambda_i$ generators. Thus, we have two possibilities:
$\V_{2,p-i}\cong\Z/(p^2)\oplus\F_p^{\lambda_i -1} $ or $\V_{2,p-i}\cong\F_p^{\lambda_i } .$

\begin{thm}
$\V_{2,p-i}\cong\Z/(p^2)\oplus\F_p^{\lambda_i -1} .$
\end{thm}
\begin{proof}
It is sufficient to find an element in $\V_{2,p-i}^{\ast}$ of order $p^2 .$

With some abuse of notations, let $\alpha^p =(q),\ q\in \Z[\zeta_0]$.
Since $\beta^p=\alpha$ in $S_{1,i}$, it follows that $\beta^{p^2}=(qt^p) ,$ where
$q\in \Z[\zeta_1] .$ We claim that the required character is defined by $(v, qt^p)_{1,1} .$
To prove this, we follow the proof of Theorem 6.1.
\begin{lemma}
$(v, qt^p)_{1,1} =\zeta_1$ for some $v .$
\end{lemma}
\begin{proof}
We have $(v, qt^p)_{1,1}^p=(v^p, qt^p)_{1,1}=(v, q)_{1,1}^p=
(v,q)_{1,0}=\newline
({\rm Norm}_{F_1/F_0}(v),q)_{0,0} .$ Clearly, we can choose $v$ such
that $({\rm Norm}_{F_1/F_0}(v),q)_{0,0} =\zeta_0$ and therefore, $(v, qt^p)_{1,1}=\zeta_1  .$
\end{proof}
\begin{lemma}
Let $r\in \Z_p[\zeta_1]$ be such that $r\equiv 1 {\rm mod} (1-\zeta_1)^{p^2+1} .$
Then $(r, qt^p)_{1,1} =1$. Further, $(\epsilon, qt^p)_{1,1} =1$ if
$\epsilon\in\Z[\zeta_1] .$
\end{lemma}
\begin{proof}
Since $r=r_1^p$, where $r_1\in\Z_p[\zeta_1]$, we can proceed exactly as in the proof of Lemma 6.4.
Since $(qt^p)=\beta^{p^2}$, again we can simply repeat the arguments of the proof of Lemma 6.4.
\end{proof}
Two lemmas above imply that the element  $qt^p$ induces a character of $\V_{2,p-i}$
of order $p^2 .$ The theorem is proved.
\end{proof}
\begin{cor}
If $b_i$ is not divisible by $p^2 ,$ then \newline
$\V_{n,p-i}\cong \Z/(p^n)\oplus
(\Z/(p^{n-1}))^{\lambda_i -1}.$
\end{cor}
\begin{proof}
Corollary 3.6 implies that $r_{m,p-i}=\lambda_i$ for any $m\geq 1 $
and moreover, the number of elements in $\V_{n,p-i}$ is
$p^{1+(n-1)\lambda_i}$. On the other hand, $\V_{n,p-i}$ is a factor
of a bigger group
$\varepsilon_{p-i}(V_{n})=\varepsilon_{p-i}(U_{n-1} /U_{n-1, p^{n} -1})$.
It is easy to verify that $\varepsilon_{p-i}(V_{n})\cong \Z/(p^n)\oplus T$,
where the abelian group $T$ has exponent $p^{n-1}$
(an exact formula can be derived from \cite{KM} but we do not need it).
Comparing the number of elements and the number of generators  of $\V_{n,p-i}$
which is
$\lambda_i$),
we can deduce that
$\V_{n,p-i}\cong \Z/(p^n)\oplus
(\Z/(p^{n-1}))^{\lambda_i -1}.$
%Therefore, the number of elements in the kernel
%of the canonical norm map $\V_{m,p-i}\to \V_{m-1,p-i}$ (which is a surjection)
%is $p^{\lambda_i}$
\end{proof}
\subsection{Fine structure of $S_{n,i}$}
Let $A$ be a finite abelian group such that $A\cong\oplus_j\ \Z/(p^k_j) .$
Let us denote the abelian group $\oplus_j\ \Z/(p^{k_j +m}) $ by $\Sigma_m A .$
\begin{lemma}
$S_{n+1,i}\cong \Sigma_n S_{1,i} .$
\end{lemma}
\begin{proof}
The fact follows from the following observations:
\begin{itemize}
  \item all the groups $S_{k,i},\ k\geq 1 ,$ have $\lambda_i$ generators;
  \item ${\rm Norm}_{F_{k+1}/F_k}(i_{F_{k+1}/F_k} (\alpha))=\alpha^p , \ \alpha\in S_{k,i} .$
\end{itemize}
{\bf Remark:} it is well-known that $i_{F_{k+1}/F_k}: S_{k,i}\to S_{k+1,i}$ is an embedding
and ${\rm Norm}_{F_{k+1}/F_k} : S_{k+1,i}\to S_{k,i}$ is a surjection.
\end{proof}

It remains to compute $S_{1,i}$.
\begin{thm}
Assume that $1\leq \lambda_i < p-1 .$ Then
$S_{1,i}\cong\Z/(p^2)\oplus\F_p^{\lambda_i -1} .$
\end{thm}
\begin{proof}
$S_{1,i}\cong\Z_p [T]/((T+1)^p -1, f_i (T)),$ see \cite{W}.
Since $p^2$ does not divide $b_i ,$ the polynomial $f_i (T)$ is irreducible.
Let $a$ be its root. Then $\Z_p [T]/(f_i (T))\cong\Z_p [a] ,$
$1,a,a^2 \cdots ,a^{\lambda_i -1}$ generate  $\Z_p [a]$ as
an abelian group, $(p)=(a^{\lambda_i}) $
in $\Z_p [a] ,$ and
$S_{1,i}\cong\Z_p [a]/((a+1)^p -1) .$
Further, $((a+1)^p -1)=(a^{\lambda_i +1})$ because $\lambda_i <p-1 .$
It follows that the element $1\in\Z_p [a]$ has exponent $p^2$ and all other
generators of $\Z_p [a]$, $a,a^2,\cdots ,a^{\lambda_i -1}$ have exponent $p.$
The theorem is proved.
\end{proof}
The case $\lambda_i =p-1$ is more delicate. To treat this case we need the
Cartesian square from Lemma \ref{S}. Let us denote the ring
$\Z_p [T]/((T+1)^p -1)]$ by $B .$ We have to study $B/(f_i (T))$.
We remind the reader that any element $b\in B$ can be written as a pair
$(c,d),\ c\in \Z_p,\ d\in\Z_p [\zeta_0] .$ In this notations $f_i (T)=(b_i, f_i (\zeta_0 -1)) .$

A simple analysis of this case shows the following result:
\begin{thm}
The element $1\in \Z_p [T]/((T+1)^p -1, f_i(T))\cong S_{1,i}$ has exponent $p^{\kappa}$ with
$\kappa=[\frac{k}{p-1}] +1 . $ Here $k=v_p (f_i (\zeta_0 -1))$, where $v_p$ is
the extension of the $p$-adic valuation on $\Z_p$ to $\Z_p [\zeta_0] .$
\end{thm}
\begin{rem}
$v_p (f_i (\zeta_0 -1))=v_p (L_p (s_0,\omega^{1-i}))$,
where $L_p$ is a $p$-adic L-function, $\omega $ is the
Teichm\"{u}ller character of ${\Z}/(p-1){\Z}$, and
$s_0$ satisfies the following equation: $(p+1)^{s_0} =\zeta_0 .$
\end{rem}
\subsection{Fine structure of the group "Local units modulo closure  of the cyclotomic units"}

Our aim is to describe the group $(U_{n,1}/C(n))^+$, where $U_{n,1}$ is the group of units
of $\Z_p [\zeta_n]$ congruent 1 modulo $(\zeta_n -1)$ and $C(n)$ is the closer of the subgroup
of the cyclotomic units. Let $G_{n,k}=Gal (F_n /F_k)$. If $A$ is a $G$-module, then $A^G$ is
a submodule of $G$-invariant elements of $A$. We begin with the following result:
\begin{thm}
The canonical map  $(U_{k,1}/C(k))^+\to (U_{n,1}/C(n))^+$ is an embedding and
$\{ (U_{n,1}/C(n))^{G_{n,k}}\}^+\cong (U_{k,1}/C(k))^+. $
\end{thm}
\begin{proof}
Let us consider the following short exact sequence:
$$0\to C(n)\to  U_{n,1}\to  U_{n,1}/C(n)\to 0 . $$

Thus, we get the corresponding exact sequence of cohomologies:
$$0\to C(n)^{G_{n,k}}\to U_{n,1}^{G_{n,k}}\to ( U_{n,1}/C(n))^{G_{n,k}} \to H^1 (G_{n,k}, C(n)).$$
Thus, to prove the theorem we have to show that $H^1 (G_{n,k}, C(n)^+ )=1$.

Let us show first that the Herbrand index $h(G_{n,k}, C(n)^+ )=1 .$ Indeed, $C(n)^+ $ is a
finite index subgroup of the group of global units $E_n .$ Its closure contains an open subgroup
of $U_{n,1}^+$ due to Leopoldt's conjecture about the closure of the group of global units, which is true
in our case. Thus, $C(n)^+ $ contains an open subgroup $X$ of $U_{n,1}^+$ as well.
Consequently, $h(G_{n,k}, C(n)^+ ) =h(X)\cdot h(C(n)^+/X)=1 $ because $h(C(n)^+/X)=1 $ since the group
$C(n)^+/X $ is finite and $h(X)=1$ because $X$ can be chosen as a projective $G_{n,k}$-module,
see \cite{CF}, chapter 6.

Therefore, it suffices to prove that $H^2 (G_{n,k}, C(n))=1.$
However, it is clear because $H^2 (G_{n,k}, C(n))=C(k)/Norm_{F_n /F_k}(C(n))=1$
because the norm is surjective on the group of cyclotomic units.

\end{proof}
Let us give another proof of the same theorem based on a lemma needed in the sequel.
Let us remind the reader that series $f_i (T)$ and polynomials $p_i (T), P_n(T)$ were defined
in Introduction. Let us define $g_i(T)=f_i (\frac{1+p}{1+T} -1)$ and the polynomial $q_i (T)$
exactly in the same way as $p_i (T)$ was defined from $f_i (T).$
It is clear that $deg (p_i) =deg (q_i) =\lambda_i.$ Furthemore, with our choice of
the polynomials $p_i, q_i$ we have $p_i (0)=f_i (0)=L_p (0,\omega^{1-i})=b_i$
and
$$q_i (0)=g_i(0)=f_i(p)=L_p (1,\omega^{1-i}):=c_i.$$
However, generally speaking $p_i (p)\neq f_i (p)$ while we only have
$v_p (f_i (p))=v_p (p_i (p))=v_p (c_i).$ As in Introduction, we denote
$(T+1)^{p^n}-1$ by $P_n (T)$ and $P_n (T)/P_k (T)$ by $P_{n,k} (T).$

\begin{lemma}
We have:

\begin{itemize}
 \item $\varepsilon_{p-i}U_{n,1}\cong \Z_p [[T]] /(P_n (T))=\Z_p [T] /(P_n (T))$;
 \item $\varepsilon_{p-i}C(n)\cong (g_i (T))/(g_i(T)P_n (T))$;
  \item $\varepsilon_{p-i} (U_{n,1}/C(n))\cong \Z_p [[T]] /(P_n (T), g_i (T))=
  \Z_p [T] /(P_n (T),q_i(T))$.
\end{itemize}
\end{lemma}
\begin{proof}
The first two items were proved in \cite{W}, chapters 13, 15.

Let us prove the third statement. It was proved in \cite{W}, chapter 15,
that $P_n$ and $g_i$ have no common roots. This implies that $(P_n)\cap (g_i)=(P_n\cdot g_i)$.
Thus,
\begin{gather*}
  \varepsilon_{p-i} (U_{n,1}/C(n))\cong \frac{\Z_p [[T]] /(P_n (T))}{(g_i (T))/(g_i(T)P_n (T))}=
\frac{\Z_p [[T]] /(P_n (T))}{(g_i (T))/(g_i(T)\cap P_n (T))}= \\
\\
\Z_p [[T]] /(P_n (T), g_i (T)).
\end{gather*}

\end{proof}
This lemma enables us to give another proof of Theorem 7.3.

Let $R=\Z_p [[T]] /(g_i (T))$ and let $M_n=R/(P_n)$ be an $R$-module.
Clearly, the lemma above shows that
$M_n = \Z_p [[T]] /(P_n (T), g_i (T))\cong \varepsilon_{p-i} (U_{n,1}/C(n)).$
Let us notice that $P_n$ is not a zero divisor in $R$ because $P_n$ and $g_i$
have distinct roots. Then multiplication by $P_{n,k}$ determines a well-defined
$R$-module map
$m: M_k\to M_n$ because $m(P_k)=P_n.$ It is clear that $m$ is an embedding.
Indeed, if $m(x)=0,$ then $P_{n,k}\cdot x=P_n \cdot y=P_{n,k}P_k\cdot y$ for some $y\in R.$
Since $P_{n,k}$ is not a zero divisor in $R, $ we conclude that $x=P_k\cdot y=0\in M_k.$

Furthermore, let us prove that $\{ y\in M_n :\ P_k\cdot y=0\}\cong M_k$. Indeed,
$P_k\cdot y=P_n\cdot z =P_k \cdot m(z)$ for some $z\in R$ and consequently $y=m(z).$ Since $m$ is
an embedding, the required statement follows.

The two statements above are equivalent to Theorem 7.3 formulated in terms
of $\Z_p [[T]]$-modules.

At this point we remind the reader that we assume that $p$ divides $b_i$. Clearly,
then $p$ divides $c_i$ and vise versa.
\begin{thm}
  \begin{enumerate}

    \item If $p^2$ divides $b_i$, then $\lambda_i =1$ and $p^2$ does not divide $c_i$;
    \item If $p^2$ does not divide $b_i$ and $2\leq \lambda_i \leq p-1$, then
    $p^2$ does not divide $c_i$;
    \item If $p^2$ divides $c_i$, then $p^2$ does not divide $b_i$ and $\lambda_i =1.$
  \end{enumerate}
\end{thm}
\begin{proof}
\begin{enumerate}
  \item We already know that $\lambda_i =1$. Then $c_i=g_i (0)=f_i (p)=b_i +pa_1+p^2 Z_1$,
  where $p$ does not divide $a_1 .$ Then the statement is clear.
  \item $c_i=b_i +p^2 Z_2$.
  \item We already know that if  $2\leq \lambda_i \leq p-1$, then $p^2$ does not divide $b_i$
  and consequently it does not divide $c_i$ as well. Hence, $\lambda_i =1$.
  Then $p^2$ does not divide $b_i$ because  $c_i=b_i +pa_1+p^2 Z_1$.

\end{enumerate}
\end{proof}
\begin{cor}
Assume that $p^2$ divides $c_i$. Then $\varepsilon_{p-i} (U_{n,1}/C(n))\cong \Z/(p^{n+l_i})$
with $l_i=v_p (c_i)$
\end{cor}
\begin{proof}
Let us perform a simple computation:
$$\varepsilon_{p-i} (U_{0,1}/C(0))\cong \Z_P [[T]]/(T, g_i (T))=\Z_p /(g_i (0))=\Z/(p^{l_i})=M_0 .$$
We already know that $\lambda_i =1$ and hence, $M_1\cong \Z/(p^{k})$ for some $k.$
we have the embedding $m: M_0\to M_1$ determined by the
formula $m(1_{M_0})=((T+1)^p /T)\cdot 1_{M_1}$. Also, we have a canonical projection
$res: M_1\to M_0 ,\ res(1_{M_1})=1_{M_0}.$ Further,
$res (m(1_{M_0}))=res (\{ ((T+1)^p -1)/T\}\cdot 1_{M_1})=p\cdot 1_{M_0} .$
This computation shows that $k=1+l_i$.
Similar computations with $m: M_1\to M_2$ and $res: M_2\to M_1$ yield
$res (m(1_{M_1}))=p\cdot 1_{M_1} $ and $M_2\cong \Z/(p^{2+l_i})$ and so on.

\end{proof}

\begin{cor}
Assume that $p^2$ does not divide $c_i$ and
$1\leq \lambda_i < p-1 $ Then
\itemize
\item $\varepsilon_{p-i}(U_{1,1}/C(1))\cong\Z/(p^2)\oplus\F_p^{\lambda_i -1} .$
\item $\varepsilon_{p-i}(U_{n,1}/C(n))\cong \Z/(p^{n+1})\oplus(\Z/p^n )^{\lambda_i -1} .$
\end{cor}
\begin{proof}
Proofs are literally the same as the proofs of the analogous statements about
the structure of the class groups in Subsection 7.2.
\end{proof}
\begin{rem}
Results of Subsections 7.2 and 7.3 show that the class groups $S_{n,i}$
and the groups $\varepsilon_{p-i}(U_{n,1}/C(n))$ in majority of cases are dual to each other.
Therefore, it is natural to conjecture that these groups are always dual to each other.
However, it follows from Theorem 7.12 that if $p^2$ divides $c_i ,$ it does not divide
$b_i$ and hence, our conjecture does not hold in this case.

Therefore, if we anyway believe in that conjecture, we have to exclude the case
$p^2$ divides $c_i =L_p (1,\omega^{1-i}) .$ Let us do this in the next section.

\end{rem}

\section{$L_p (1,\omega^{1-i}),\ L_p (0,\omega^{1-i})$  are not divisible by $p^2 \ ?$}
\subsection{$\V$-duality}

In this subsection we will prove the following result.
\begin{thm}
Let us assume that $\lambda_i =1$ and $p^2$ does not divide $b_i$.
Then the group $Gal(F_1 /F_0)$ acts non-trivially on $\V_2^+ .$
\end{thm}
\begin{rem}
In the definition of $\V_2$ (given by Kervaire and Murthy in \cite{KM}) has
a natural structure a $Gal(F_2 /F_0)$-module. However, they proved
that $Gal (F_2 /F_1)$ acts on $\V_2$ trivially. Consequently, $\V_2$
is a $Gal (F_1 /F_0)$-module.
\end{rem}
\begin{proof}
Our proof is based on the main result of \cite{KM} (mentioned earlier in the paper, Theorem 1.1,
however we need it in its complete form because it traces the action of $Gal(F_2 /F_0)$
or equivalently $Gal(F_1 /F_0)$):
\begin{itemize}
  \item Let $g$ be a generator of $G=Gal(F_1 /F_0)\cong \Z/p \Z$
  and let us define the following natural action of $G$ on $\V_2^*$
  $(g\chi) (v)=g(\chi (g^{-1}v)),$ where $\chi\in \V_2^* , v\in\V_2 .$
  \item Then $(\V_2^+)^*$ is isomorphic to a $G$-submodule of $S_1^{(p)} .$
  \item Assuming the Kervaire and Murthy conjecture that
  $S_1^{(p)} \cong (\V_2^+)^*\cong \Z/p^2\Z$
  we get the formula $<gs,v>=g(<s,g^{-1}v>)$
  or equivalently $g(<g^{-1}s,v>)=<s,gv>.$ Here $s\in S_1^{(p)} .$
\end{itemize}
Now, let us proceed to the proof of our theorem.
Clearly, without loosing generality we may assume that
$S_1^{(p)} =S_{1,i} .$

Then, since
$\lambda_i =1$ and $p^2$ does not divide $b_i$, our previous computations
show that $S_1^{(p)}\cong\V_2^+\cong\Z/p^2\Z ,$ i.e. the Kervaire and Murthy
conjecture is true in our situation.

Let us choose $g\in G$ such that $g(\zeta_1)=\zeta_1^{1+p}$.
It is well-known due to Iwasawa that the subgroup
$\{s\in  S_1^{(p)}:\ g(s)=s  \}\cong S_0^{(p)}\cong \Z/p\Z .$
Therefore, $G$ acts on $S_1^{(p)}$ non-trivially and we can choose a generator
$s\in S_1^{(p)}$ such that $g^{-1}(s)=s^{1+p}$. Further, we can choose a generator
$v\in \V_2^+$ such that $<s,v>=\zeta_1.$

Suppose $G$ acts on $\V_2^+$ trivially. Then
$$\zeta_1 =<s,v>=<s,g(v)>=g(<g^{-1}(s),v >)=g(\zeta_1^{1+p})=\zeta^{1+2p} . $$
Clearly, it is impossible and hence, $G$ acts on $\V_2^+$ non-trivially.

\end{proof}
\subsection{$E$-duality}
Let us assume that $\lambda_i =1$ and $p^2$ does  divide $c_i$.

Results of this subsection are based on Chapter 8 of \cite{W}
Let $E$ be the group of real units of $\Z [\zeta_0] $ and in this subsection
we denote $U_{0,1}^+$ by $U .$ Since $p$ satisfies Vandiver's conjecture,
$U/(C(0))^+ = U/\bar{E}$, where $\bar{E}$ is the closure of $E$ in $U.$

Let us denote $E/E^{p^2}$ by $E_{p^2}$. It was proved in \cite{W} that
$\varepsilon_{p-i}   E_{p^2}\cong \Z/(p^2)$ and
$\varepsilon_{p-i}   U/\varepsilon_{p-i} \bar{E}\cong \Z/(p^{v_p (c_i)})=\Z_p /(c_i).$
according to Corollary 7.13.

Let $\eta\in E$ generate $\varepsilon_{p-i} E_{p^2}$. Let us consider
$\varepsilon_{p-i} (\eta)\in \varepsilon_{p-i} \bar{E}$. Since
$\varepsilon_{p-i} (E/E^{p^2}) =\varepsilon_{p-i} \bar{E}/\varepsilon_{p-i}\bar{E}^{p^2}$ and
$E^{p^2}\subset\bar{E}^{p^2}$, we see that
$\eta^{-1}\varepsilon_{p-i}(\eta)\in \bar{E}^{p^2} .$

Thus $\eta =\varepsilon_{p-i}(\eta)\gamma^{p^2}$ and we obtained the following
\begin{lemma}
$\eta\in E$ is a local $p^2$-power.
\end{lemma}

\begin{proof}
$\varepsilon_{p-i}(\eta)$ is  a local $p^2$-power because
$\varepsilon_{p-i}(\eta)\in\varepsilon_{p-i}\bar{E}=(\varepsilon_{p-i}   U)^{v_p (c_i)}$
and $v_p (c_i)\geq 2.$ Thus, $\eta\in E$ is a local $p^2$-power.

\end{proof}
\subsection{$L_p (1,\omega^{1-i})$  is not divisible by $p^2 \ ?$}
Now we can prove the first main result of this section
\begin{thm}
$L_p (1,\omega^{1-i})$ is not divisible by $p^2$.
\end{thm}
\begin{proof}
Assuming that $L_p (1,\omega^{1-i})$ is  divisible by $p^2$, in several steps we
will come to contradiction.

\begin{itemize}
  \item Clearly, $\eta$ is not a $p$-power in the field $F_0$ because it
  generates $\varepsilon_{p-i} (E/E^{p^2}).$
  \item $\eta$ is not a $p$-power in the field $F_1$ because otherwise
  $F_1=F_0 (\eta^{1/p})$ and $F_1$ becomes a non-ramified extension of $F_0 .$
  \item Hence, the extension $F_1 (\eta^{1/p^2})/F_1$ is non-ramified and consequently
  $\eta$ induces a character of $S_1^{(p)} .$
  \item Let $E_1$ be the subgroup of the group of units of $\Z [\zeta_1]$ which are local
  $p^2$-powers. Then $E_1$ generates a subgroup of the group of characters of $S_1^{(p)} .$
  The corresponding pairing between $S_1^{(p)} $ and $E_1$ satisfies
  $< gs, \epsilon> =g(< s, g^{-1}\epsilon>)$.
  \item Let $X\subseteq (S_1^{(p)})^*\cong \Z/(p^2)$ be the subgroup generated by $\eta .$
  Comparing formulas  $< gs, \epsilon> =g(< s, g^{-1}\epsilon>)$ and
  $< gs, v> =g(< s, g^{-1}v>)$ we can conclude that
  $X\cong \V_2^+$ as $G$-modules.
  \item However, the latter is impossible because $G$ acts trivially on $X$
  ($\eta\in F_0$) and non-trivially on $\V_2^+$.
\end{itemize}
The theorem is proved.
\end{proof}
\subsection{$L_p (0,\omega^{1-i})$  is not divisible by $p^2 \ ?$}
Now we can prove the second main result of this section
\begin{thm}
$L_p (0,\omega^{1-i})$ is not divisible by $p^2$.
\end{thm}
\begin{proof}
Assuming that $L_p (0,\omega^{1-i})$ is  divisible by $p^2$, in several steps we
will come to contradiction.

 We remind the reader that we already know that $\lambda_i=1$ and $c_i$
 is divisible by $p$ but not by $p^2 .$ Let us assume first that
 $v_p (L_p (0,\omega^{1-i}))=v_p (b_i)=2 .$ The reader will see that the general
 case will be possible to treat exactly in the same way.

\begin{itemize}
  \item Since $S_{0,i}\cong\Z/(p^2)$ and $S_{1,i}\cong\Z/(p^3)$ under
  our assumptions, we deduce that $S_{1,i}^G\cong\Z/(p^2)=S_{1,i}^p .$
  \item Our previous computations show that
  $\varepsilon_{p-i} \V_1^+\cong \F_p\cong \varepsilon_{p-i}(U_{0,1}/C(0))$ and
  $\varepsilon_{p-i} \V_2^+\cong \Z/(p^2)\cong \varepsilon_{p-i}(U_{1,1}/C(1))$.
  \item Since there exists a canonical projection of $G$-modules $U_{1,1}\to \V_2^+$,
  we deduce that $\varepsilon_{p-i} \V_2^+\cong \varepsilon_{p-i}(U_{1,1}/C(1))$
  as $G$-modules.
  \item $(U_{1,1}/C(1)))^G\cong U_{0,1}/C(0))$ by Theorem 7.10, we see that
  $G$ acts on $\varepsilon_{p-i} \V_2^+$ non-trivially.
  \item By the main Kervaire--Murthy theorem (\cite{KM}),
  $(\varepsilon_{p-i} \V_2^+)^*$ is isomorphic to a subgroup of $S_{1,i}$ as $G$-modules.
  Comparing orders of the involved groups we get $(\varepsilon_{p-i} \V_2^+)^*\cong (S_{1,i}^G$
  as $G$-modules.
  \item Let us fix $g\in G$ such that $g(\zeta_1)=\zeta_1^{1+p}$.
  Let us choose $v\in \varepsilon \V_2^+$ so that $g^{-1}(v)=v^{1+p}$
  and $s\in S_{1,i}^G$ satisfying $<s,v>=\zeta_1 .$
  \item Now, as before let us perform a simple computation:
  $$ <s,v>=\zeta_1= <gs,v>=g<s,g^{-1}(v)>=g(\zeta_1^{1+p})=\zeta_1^{1+2p}.$$
\end{itemize}
The latter equality is impossible and the theorem is proved.
\end{proof}

\subsection{Correction to the previous results of this section }

The results of this section must be corrected. What went wrong?
Let us take a look at the proof of Theorem 8.1. We chose a generator
of $g\in G$ such that $g(\zeta_1 )=\zeta_1^{1+p}$ and and a generator
$s\in S_{1,i}$ such that $g^{-1} (s)=s^{1+p} .$

Unfortunately, these two choices might be incompatible. 
We can only claim that $g(s)=s^{1+kp} $ with an integer $k$ defined
modulo $p .$ Similarly, $g(v)=v^{1+lp} $, where $v\in\V_2^+$ satisfies
$<s,v>=\zeta_1 .$ 

Further, since $<g(s),g(v)>=g(\zeta_1 ),$ we see that $k+l=1 .$
Assuming that $L_p (1,\omega^{1-i})$ is  divisible by $p^2$, we can deduce
that $l=0\ {mod} ( p)$ and hence $g(s)=s^{1+p} .$

Let us consider $S_{1,i}$ as an $\Z_p [[T]] /(f_i (T),(1+T)^p -1)$-module.
It follows that $T-p$ acts trivially on $S_{1,i}\cong\Z/(p^2)$. Since $\lambda_i =1$,
we infer that $f_i (T)$ has a root of the form $T=p+ap^2$ for some $a\in\Z_p .$

Furthermore, $T=(1+p)^s -1$ and hence $s_0 =log_{1+p} (1+p+ap^2 )=1+bp$ is a zero
of $L_p (s,\omega^{1-i})$. We have proved
\begin{thm}
Let $\lambda_i =1$.
If $L_p (s,\omega^{1-i})$ has no zeroes of the form 
$$s_0 =log_{1+p} (1+p+ap^2 )=1+bp ,\ b\in\Z_p,$$
then  $L_p (1,\omega^{1-i})$ is not divisible by $p^2$. 
\end{thm}
The case $k=0, l=1$ can be conditionally excluded  using a similar result.
\begin{thm}
Let $\lambda_i =1$.
If $L_p (s,\omega^{1-i})$ has no zeroes of the form
$$s_0 =log_{1+p} (1+ap^2 )=bp ,\ b\in\Z_p,$$
then  $L_p (0,\omega^{1-i})$ is not divisible by $p^2$. 
\end{thm}
\begin{rem}
The previous conditional theorem proves the Kervaire and Murthy conjecture (conditionally)
with only one possible exception: $\lambda_i=p-1.$
\end{rem}

\section{Concluding remarks}

The Kervaire and Murthy conjecture has another interesting form. Let us
denote by ${\A}(F_{n})$ the ring of adeles of the field $F_{n}$. Let $w$ be
a valuation of $F_{n}$, different from $\mu_{n}=(1-\zeta_{n})$. Let
$\mathbb{Q}_{w}$ be the completion of ${\Z}[\zeta_{n}]$ at $w$. Let us
consider the following subgroup $K_{p^{n+1}-1}$ of $GL(1, {\A}(F_{n}))$,
namely
$$
K_{p^{n+1}-1}=GL(1, \mathbb{Q})\times U_{n,p^{n+1}-1}\times \prod
GL(1,\mathbb{Q}_{w}).
$$
Then the Kervaire and Murthy conjecture can be formulated as

\begin{con}
$(S^{-}_{n})^{\ast} \cong \{ GL(1,F_{n}) \setminus GL( 1,{\A}(F_{n})) / K_{p^{n+1}-1}\}^{+}_{(p)}$
\end{con}

%After our all computations in the paper,  the Kervaire and Murthy conjecture has been proved
%in all the cases with only one exception: $\lambda_i =p-1,$ and hence  $v_p(L_p %(0,\omega^{1-i}))=1.$

%Our results have many interesting consequences. Let me mention some of them.
%\begin{itemize}
  %\item Kummer's Lemma: If $\epsilon\in\Z[\zeta_0]$ is a unit, which is congruent to a
 % rational integer $\rm{mod}$  ${p^2}$, then it is a $p$-power of a unit.
  %\item Let $r_0=\#\{B_k,\ k=2,4,..,p-3: v_p(B_k)>0  \}$. Here $B_k$ are Bernoulli numbers.
  %Then the group  $(Cl(F_0))^{(p)}\cong\F_p^{r_0}$.
  %\item $Res_{s=1} \zeta_{F_0^+, p} (s)=(p^{r_0}-p^{r_0-1})\cdot u$, where $u$ is a $p$-adic unit,
 % see \cite{W}.
  %\item However a main possible consequence -- another proof of Fermat's Last Theorem -- seems %still to be unreachable.
%\end{itemize}
\section{Afterword}
This is my last paper before my retirement from the Department of Mathematical Sciences,
Chalmers University of Technology/ University of Göteborg, Sweden after almost 25 years of
professional service. I appreciated and enjoyed very much its friendly and calm atmosphere, which encouraged
and helped
me to work on the problems of my interest "i lugn och ro."

\end{document}